\pdfoutput=1
\documentclass[12pt]{amsart}
\usepackage{geometry}                
\usepackage{graphicx}
\usepackage{amssymb}
\usepackage{amsthm}
\usepackage{dsfont}
\usepackage{comment}

\newtheorem{prop}{Proposition}
\newtheorem{thm}{Theorem}

\theoremstyle{remark}
\newtheorem{rmk}{Remark}

\usepackage[linesnumbered,ruled,vlined]{algorithm2e}
 \SetKwInput{KwInput}{Input}                
\SetKwInput{KwOutput}{Output}
\usepackage[backend=bibtex,style=authoryear,natbib=true, url=false,isbn=false, doi=false, eprint=false]{biblatex} 
\renewbibmacro{in:}{}
\providecommand{\reals}{\mathbb{R}}
\providecommand{\eps}{\varepsilon}

\providecommand{\expec}{\mathbb{E}}

\newcommand{\tr}{\operatorname{tr}}
\providecommand{\oh}{\mathrm{o}}
\providecommand{\Oh}{\mathrm{O}}
\providecommand{\normal}{\mathcal{N}}
\providecommand{\KL}{\operatorname{KL}}
\providecommand{\1}{\mathds{1}}
\newcommand{\argmin}{\operatornamewithlimits{\arg\min}}

\title[Self-ROT and GMM MLE]{ Regularised Optimal self-Transport is approximate   
  Gaussian Mixture Maximum Likelihood}
\author{G. Mordant}
\thanks{The author gratefully acknowledges the support of the DFG through the CRC 1456.}
\address{
\parbox{\linewidth}{IMS,Universität Göttingen\\
Goldschmidtstra\ss e, 7  \\
37073, Göttingen, Germany }
 }
\email{gilles.mordant@uni-goettingen.de}
\date{\today}                                           

\begin{document}
\begin{abstract}
We investigate the link between regularised self-transport problems and maximum likelihood estimation in Gaussian mixture models (GMM). 
This link suggests that self-transport followed by a clustering technique leads to principled estimators at a reasonable computational cost. 
Also, robustness, sparsity and stability properties of the optimal transport plan arguably make the regularised self-transport a statistical tool of choice for the GMM.
\end{abstract}
\maketitle

\section{Introduction}

Clustering is a standard problem in statistics and data science. Given a set of observations, the aim is to separate these observations in $K$ groups that are similar. 
A typical model for which clustering is required is that of Gaussian mixtures; its definition is recalled in Section~\ref{sec: GausMix}. 
Even though this model has a long history, answering the natural questions pertaining to the estimation of its parameters are still an active research endeavour, see \citet{wu2021randomly,doss2023optimal,loffler2021optimality} and the references therein for a more complete overview.

In a recent line of work \citet{marshall2019manifold}, \citet{landa2021doubly}, \citet{landa2023robust} investigated the use of entropy-regularised self-transport to perform manifold learning;
see Section~\ref{sec: Reg} for an introduction to regularised optimal transport. 
In a subsequent work, \citet{zhang2023manifold} considered the impact of using quadratically regularised optimal transport for the same purposes, stressing the importance of the sparsity of the obtained transport plan for local adaptivity.  In that last work, the authors showed the usefulness of applying spectral clustering on the optimal transport plan to recover the clusters of Gaussian Mixtures Models. The usefulness of bistochastic projection was also independently observed by \citet{lim2020doubly}. In that work, however, they do not consider the projection onto hollow bistochastic matrices, an important feature as it turns out. Also, the statistical point of view is not present in that latter reference.

In this note, we exhibit the reason why such a method performs well in practice by relating it to maximum likelihood estimation of GMM. This is done in Section~\ref{sec: MainCont}.
Before doing so, we first provide preliminary background.

\subsection{Gaussian mixtures}
\label{sec: GausMix}
Let us denote by $\Delta_K$ the $K$-simplex\footnote{The $K$-simplex is defined as  $\Delta_k:=\big\{x \in \reals_+^K : \sum_{k=1}^K x_k =1\big\}$.}.
We recall that $d$-dimensional Gaussian mixtures model with $K$ components can be thought of as the result of two-stage process. First, one selects the cluster or subgroup at random, which is modelled as sampling from the multinomial $\operatorname{Mult}(1; \theta_1,\ldots, \theta_K)$, where $\theta \in \Delta_K$ is the vector encoding the group proportions. 
Then, conditionally on the selected cluster, one draws from the Gaussian random variable with corresponding parameters. We denote the  mean vectors by $\mu_k $ and the covariance matrices by $\Sigma_k$, $k\in \{1, \ldots K\}$. For now, let us further set $\Sigma_k= \sigma^2 I_d$, for each $k$ and assume that $\sigma^2$ is fixed and known. The cluster membership of the observation $i$ is encoded in the  random vector $Z_i$ of size $K$ containing only zeros except at the position of the cluster where it is one.

The observations thus are pairs $(Z_i, Y_i), 1 \le i \le n$ and maximising 
the log-likelihood is equivalent to maximising the function 
\begin{equation}
\label{eq: LogLik}
\ell_n(Z, \mu):= \sum_{i=1}^n \sum_{k=1}^K Z_{ik}\left( \ln \theta_k -  \frac1{2\sigma^2}\| Y_i - \mu_k \|^2\right) - n d\log \sigma.
\end{equation}
Solving this maximisation problem is challenging, in particular because the $Z_i$'s are not observed. We now recall some basics of discrete, regularised optimal transport.

\subsection{Regularised optimal transport}
\label{sec: Reg}
In (discrete) optimal transport, one is given two sets of points $\mathbb{X}:=\{x_1, \ldots, x_n\}$ and $\mathbb{Y}:=\{y_1, \ldots, y_n\}$. These sets are endowed with a (probability) measure, which we'll take to be uniform\footnote{The problem naturally generalises to non-uniform measures and $\mathbb{X},\mathbb{Y}$ with different cardinalities. We do not require this here, though.}. Given a cost function $c : \mathbb{X} \times \mathbb{Y}\to \reals$, there is a corresponding $n\times n$ cost matrix $C$. The regularised optimal transport problem then reads, 
\[
	\min_{\pi \in B_{n}} \langle \pi , C \rangle + \eps \Phi (\pi),
\]
where $B_{n}$ is the set of bistochastic matrices, $\eps$ is a regularisation parameter and $\Phi$ is a convex function. Usually, a $f$-divergence with respect to the uniform measure on $\mathbb{X}\times \mathbb{Y}$ is chosen. 
In this note, we consider the two most classical functions $\Phi$.
\begin{itemize}
\item First, $\Phi(\pi)= \sum_{i,j} \pi_{i,j} \log( \pi_{i,j})$, with the convention $ 0 \log 0=0$. The problem in \eqref{Eq: RegOT} for that particular choice is then called the entropy-regularised optimal transport.
\item Second, the choice $\Phi(\pi)= \sum_{i,j} \pi_{i,j}^2/2$, gives birth to the quadratically regularised optimal transport problem.
\end{itemize}

A classical choice for the cost function, provided that the points belong to $\reals^d$, is the squared Euclidean distance. A reason for that is that the theory of optimal transport with squared distance is particularly elegant (in the continuous case); we refer to \citet{villani2008optimal} for a complete account of that theory.

Importantly, the problem admits a dual form, which is often more convenient for numerics. The dual formulation reads, 
\[
\sup_{u,v \in \reals^n } u^\top 1_n + v^\top 1_n - \eps \Phi^*\big( u 1_n^\top + v 1_n^\top  - C\big),
\]
where $\Phi^*$ is the Fenchel--Legendre transform of $\Phi$ and $1_n$ is a $n$-dimensional vectors of ones. Note that $\Phi$ is a function of matrices with positive entries, implying that the supremum in the definition of the Fenchel--Legendre transform is restricted. One thus gets in the quadratically regularised case
\[
 \Phi^*( u 1_n^\top + v 1_n^\top  - C) =\frac12 \left(  \big[ u 1_n^\top + 1_n  v^\top  - C]_+ \right)^2,
\]
where $[x]_+$ is the positive part of $x$.
In both cases, the solution of the primal and the dual problem, $\pi^*$ and $u^*, v^*$ are related by the equalities
\[
\pi_{\KL}^* = \exp\left( \frac1\eps \big( f^* 1_n^\top +  1_n (g^*)^\top  - C\big)\right)\quad  \text{or}\quad  \pi_{Q}^*=\frac1\eps \left(  \big[  u^* 1_n^\top +  1_n (v^*)^\top  - C]_+ \right), 
\]
where $(f^*,g^*)$ and  $(u^*, v^*)$ are the pairs of optimal dual potentials of the respective problems. One sees from that formulation that each entry of the optimal plan for the entropy-regularised case will be positive, while the one for the quadratic regularisation can (and usually does) have entries equal to zero.

Finally, we'd like to stress an important property of entropy or quadratically regularised optimal transport, when $B_n$ is replaced by the set of bistochastic matrices with zero diagonal, call the latter set $B_n^\circ$.
In that case, the problem exhibits robustness to \emph{heteroskedastic} noise in high dimensions. 
For more clarity, consider a set of points $x_i$'s in high dimension that  is corrupted with noise, i.e., one observes 
\[
\tilde x_i = x_i + \eta_i,
\] 
where $\expec \eta_i = 0$ and $\expec \eta_i \eta_i ^\top = \Sigma_i$, where the $\Sigma_i's$ can be different for all indices. 
In high dimensional settings, for $i\neq j$ and well-behaved\footnote{See \citet{landa2021doubly} or \citet{zhang2023manifold}  for a precise statement of the conditions. In essence, they mean that the noise should be spherical enough and not concentrate in certain directions.} $\Sigma_i$'s, it holds that 
\[
\| \tilde x_i - \tilde x_j \|^2 \approx \expec \| \eta_i\|^2 + \expec \| \eta_j\|^2  + \|  x_i -  x_j \|^2.
\]
Therefore, if one enforces the transport plan to be zero on the diagonal, the additional terms arising from the heteroskedastic noise are a rank one perturbation, which is absorbed by the optimal potentials. We will state a similar result adapted to the object of this paper in Section~\ref{sec: Main}, which is why we only sketch the previous result.
Therefore, the self-transport problem that we will be considering is 
\begin{equation}
\label{Eq: RegOT}
\pi^*_{\eps,\Phi}=\argmin_{\pi \in B_{n}^\circ} \langle \pi , C \rangle + \eps \Phi (\pi),
\end{equation}

Before moving to the next section, we stress the fact that quadratically optimised transport is equivalent to a metric projection problem for $\eps>0$, i.e., 
\[
\pi_{\eps,Q}^*(C):= \argmin_{\pi \in B_n^\circ } \langle C, \pi \rangle + \frac{\eps}{2} \|\pi\|_F^2 =  \argmin_{\pi \in B_n^\circ } \left \| \pi - \frac{-C}{\eps} \right\|_F^2. 
\]
As such, $C\mapsto\pi_{\eps,Q}^*(C)$ is thus a Lipschitz function of $C$, implying that small perturbations of $C$ won't have a dramatic impact on the solution. 

\section{Main results}
\label{sec: MainCont}
We now turn to the main contributions of this note.
In the sequel, $H(\theta)$ denotes the entropy of the non-negative vector $\theta$ and $\|A\|_2$ denotes the largest singular value of the matrix $A$.

\subsection{Isotropic Gaussain mixtures}
\label{sec: IsoGausMix}
\begin{thm}
\label{thm: main}
Consider a sample of an isotropic Gaussian mixture model $(Y_i)_{i=1}^n$ with known variance $\sigma^2$ as above. Set 
\[
C_{ij}:=\frac12\| Y_i - Y_j \|^2.
\]
and assume that $\theta$ and $K$ are known.
Then, the problem \eqref{Eq: RegOT}
is asymptotically equivalent to 
\begin{equation}
\label{eq: ApproxLogLik}
\max_{Z\in \{0,1\}^{n\times K}} \sum_{i=1}^n \sum_{k=1}^K Z_{ik} \ln \theta_k  - \sum_{i=1}^n Z_{ik} \left(\|Y_i -\mu_ k \|^2 + O_p\left(\tfrac{ \sqrt{d}}{\sqrt{n}}\right)\right)- nd(\log \sigma- \sigma^2),
\end{equation}
 for the choice 
\[
\eps = \frac{n\sigma^2}{K} H(\theta).
\]
\end{thm}
The proof is deferred to Section~\ref{sec: Main}.
\begin{rmk}
As announced in the title, this theorem means that quadratically regularised (hollow) self-transport approximates maximum likelihood estimation in a Gaussian mixture model for a proper choice of the regularisation parameter $\eps$. 
\end{rmk}

Naturally, the vector $\theta$ and $K$ are unknown, but the order of magnitude of the entropy might be roughly known. For the parameter $K$, it is important to note that
 the optimal solution $\pi^*$ of  \eqref{Eq: RegOT} is such that 
\[
\pi_{ij}^* \approx \sum_{k=1}^K \1_{i\neq j}Z_{ik}Z_{jk}(N_k-1)^{-1},
\] 
where the approximate equality sign is a consequence of the error terms in \eqref{eq: ApproxLogLik}. 
The number of clusters will thus reflect in the rank of  $\pi^*$. 
Altogether,  a clustering method applied to $\pi^*$ is a thus natural way to recover the different components of the mixture. 
From these clusters, one can then estimate $K$, $\theta$ and $\sigma$ and verify that these values are in line with the optimal parameter $\epsilon$.

\begin{rmk}[Unknown $\sigma^2$ and $\theta$]
In practice, one might want to estimate $\sigma^2$ and $\theta$ as well. 
One could thus add the term $- nd(\log \sigma- \sigma^2)$ to the optimal transport problem with the choice of the regularisation parameter as above. 
One could then perform a grid search to find (approximately) optimal $\sigma^2$ and $\eps/\sigma^2$. 
\end{rmk}

We leave the practical questions and numerical experiments pertaining to this approach for future works.

\subsection{Robustness to heteroskedasticity}

The results of Section~\ref{sec: IsoGausMix} holds under the assumption that the Gaussians at play are isotropic with variance parameter $\sigma^2$.
As already hinted at in Section~\ref{sec: Reg}, quadratically regularised hollow self-transport behaves well under heteroskedastic noise. 
This enables to extend the usefulness of the approach to cases of the different $\Sigma_k$'s, provided that there exists $\sigma^2>0$ and a decomposition 
\begin{equation}
\label{eq: DecompSigma}
\Sigma_k = \sigma^2 I_d + \tilde \Sigma_k,  \qquad k \in \{ 1, \ldots, K\},
\end{equation}
with each $\tilde \Sigma_k$ being positive semi-definite.
This condition is violated if at least one $\Sigma_k$ is merely positive \emph{semi}-definite, it is thus not very restrictive.
Under this decomposition, we can see the observations $(Y_i)_{i=1}^n$ as given by 
\[
Y_i = \tilde Y_i + \eta_i,
\]
where the $Y_i$'s are observations of an isotropic GMM and the noise vectors $\eta_i$ are such that 
\[
\eta_i \vert Z_{i,k}   \sim \normal \left(0,  \sum_{k=1}^K  Z_{i,k} \Sigma_k \right).
\] 
We then define $C$ and $\tilde C$ to be the pairwise squared Euclidean matrices of the $Y_i$'s and the  $\tilde Y_i$'s, respectively.
With this notation, we can state the following theorem. 

\begin{thm}
\label{thm: Rob}
Take the heteroskedastic GMM as introduced above.
Assume that there exist a constant $C_\mu$ such that 
$
\max_{1\le k,k' \le K} \| \mu_k -\mu_{k'} \|\le C_\mu.
$
Further assume that $\| \tilde \Sigma_k\|_2 \le C_\Sigma, \forall k \in \{1, \ldots, K \}$. Then, 
\[
\big\|\pi_{\eps,Q}^*\left( C \right) - \pi_{\eps,Q}^*\left(C\right) \big\|_F \le  \frac{1}\eps \big\| E\big\|_F,
\]
where $E$ is a centred error matrix for which each entry is  $\Oh_p\left(d^{1/2}\right)$.
\end{thm}

\begin{rmk}
Remark that the order of the entries of the cost matrix is $\Oh_p(d)$, so that the errors indeed have a limited impact. We further stress that the right-hand side in the inequality above is an upper bound; it is difficult to assess whether it is sharp.
A closer inspection of the proof of Proposition~\ref{prop: Robust} shows that more assumptions on the eigendecomposition of the $\tilde \Sigma_k$'s coupled with conditions on the direction of the vectors 
$ \mu_k - \mu_{k'}$, $1\le k,k'\le K$, could guarantee better bounds.
\end{rmk}
\begin{proof}[Proof of Theorem~\ref{thm: Rob}]
Proposition~\ref{prop: Robust} below states that $C - \tilde C$ is the sum of a rank-one matrix, a diagonal matrix and some stochastic matrix that is $\Oh_p\left(d^{1/2}\right)$.
Remark that a rank one-one perturbation and a diagonal perturbation of the cost matrix leaves the optimiser of \eqref{Eq: RegOT} invariant.
Thus, the difference between the solutions for $C$ and $\tilde C$ only depends on the $\Oh_p\left(d^{1/2}\right)$ matrix. 
The Lipschitzianity of the metric projection recalled in Section~\ref{sec: Reg} yields the claim.
\end{proof}


\section{Proofs}
\label{sec: Main}

\subsection{Proof of Theorem 2}
An estimator of \eqref{eq: LogLik} that removes the problem of estimating the $\mu_k$'s \textemdash even though still intractable as the $Z_i$'s are unknown \textemdash is 

\begin{align}
\label{eq: AltLogLik}
\nonumber
&\tilde \ell_n(Z,\theta):=\\
&\quad \sum_{i=1}^n \sum_{k=1}^K Z_{ik} \ln \theta_k  - \sum_{i=1}^n\sum_{j=1}^n \sum_{k=1}^K \frac{Z_{ik}Z_{jk}}{N_k-1 } \1_{i\neq j}\frac1{2\sigma^2} \| Y_i - Y_j \|^2 - nd(\log \sigma- \sigma^2),
\end{align}
where $N_k$ is the unobserved number of elements in the $k$th cluster.
We first establish that this estimator is actually fulfilling its purpose.

\begin{prop}
\label{prop: error}It holds that
\[
\|Y_i -\mu_ k \|^2 = \sum_{j=1}^n \frac{Z_{jk}}{N_k-1 } \1_{i\neq j}\frac1{2\sigma^2} \| Y_i - Y_j \|^2 - d \sigma^2  + O_p\left(\frac{ \sqrt{d}}{\sqrt{n}}\right).
\]
\end{prop}

For comparison of the errors magnitude, notice that $\|Y_i -\mu_ k \|^2 = O_p\left( d \right)$.

\begin{proof}
An expansion of squares yields
\[
\| Y_i - \mu_k \| ^2 =\| Y_i - Y_j \|^2  -2 \langle  Y_i - \mu_k,  \mu_k - Y_j \rangle  -  \|  \mu_k - Y_j \|^2, 
\]
Note that, for $j=i$, the error is  $\|  \mu_k - Y_i \|^2$, which has a different sign from that of  $-  \|  \mu_k - Y_j \|^2$ for the other cases.
This explains the presence of the indicator variable. 
Then, basic computations yield
\[
 \sum_{\substack{j=1\\ j \neq 1}}^n \frac{Z_{jk}}{{N_k-1 }} \langle  Y_i - \mu_k,  \mu_k - Y_j \rangle \Big\vert (Y_i , Z_1, \ldots, Z_n) \sim \frac{1}{\sqrt{N_{k}-1 }} \normal\Big(0, \sigma^2 \sum_{\ell=1}^d(Y_{i,\ell}- \mu_{k,\ell})^2\Big )
\]
as well as 
\[
 \sum_{\substack{j=1\\ j\neq i}}^n \frac{Z_{jk}}{{N_k-1 }} \|  \mu_k - Y_j \|^2 \Big\vert ( Z_1, \ldots, Z_n) \sim \frac{\sigma^2}{N_k-1} \chi_{d\times (N_k-1)}^2. 
\]
The claim follows from a straightforward use of the laws of total expectation and variance.
\end{proof}

Proposition~1 formalises the intuition that maximising the alternative \eqref{eq: AltLogLik} instead of \eqref{eq: LogLik} come at little cost but removes the problem of estimating the $\mu_k$'s. 
To simplify notation, let us now set  $\pi_{ij} :=\sum_{k=1}^K \1_{i\neq j}Z_{ik}Z_{jk}(N_k-1)^{-1}$. 

\begin{prop}
\label{prop: MatPi}
The matrix $\pi$ is a symmetric bistochastic matrix with all diagonal elements equal to zero.
\end{prop}
\begin{proof}
One easy sees that $\pi$  is a (sparse) matrix with positive entries.
The indicator in the definition ensures that the matrix is hollow. 
 Further, for each $ i \in\{1, \ldots,n\}$, 
\[
   \sum_{j=1}^n  \sum_{k=1}^K \1_{i\neq j}Z_{ik}Z_{jk}(N_k-1)^{-1} =    \sum_{k=1}^K (N_k-1)^{-1} Z_{ik}\sum_{j=1}^n  \1_{i\neq j} Z_{jk} = 1.\qedhere
\]
\end{proof}

With the notation set before, Equation~\eqref{eq: AltLogLik} can thus be rewritten as
\begin{align*}
 & \sum_{k=1}^K \sum_{i=1}^n Z_{ik} \ln \theta_k - \frac{1}{\sigma^2} \langle \pi , C \rangle - nd(\log \sigma- \sigma^2)\\
&\qquad \qquad =   \sum_{k=1}^K  N_k  \ln \theta_k  - \frac{1}{\sigma^2} \langle \pi, C \rangle - nd(\log \sigma- \sigma^2).
\end{align*} 
In view of Equation~\eqref{Eq: RegOT}, thus remains to relate $\sum_{k=1}^K  N_k  \ln \theta_k$ with the regularisers of the OT problem. 
First, in the simple case where all clusters have the same probability, i.e., if  $\theta_1 = \cdots = \theta_K$, 
\[
\sum_{k=1}^K N_k \log \theta_k = - n \log K.
\]
In general, as $N_k = n \theta_k \left( 1+ \Oh_p\big( \sqrt{(1-\theta_k)/(n\theta_k)}\big)\right)$
\[
\sum_k N_k \log \theta_k = n   \sum_k \theta_k \log \theta_k + \Oh_p(\sqrt{n}).
\]
Remark that
\begin{align*}
 \|\pi\|^2_F&=\sum_{i,j} \left( \sum_k  Z_{ik}Z_{jk} \1_{i\neq j} (N_k -1)^{-1} \right)^2 \\
 &= \sum_{i,j}\sum_k  Z_{ik}Z_{jk} \1_{i\neq j} (N_k -1)^{-2} =  \sum_{i}\sum_k  Z_{ik}(N_k -1)^{-1} \\
 &=\sum_k N_k (N_k -1)^{-1} = K + \Oh_p(\sum_k n^{-1} \theta_k^{-1})
\end{align*}
Therefore, while maximising the approximate surrogate likelihood is the maximisation of
\[
\pi \mapsto n   \sum_k \theta_k \log \theta_k + \oh_p(n) - \frac{1}{\sigma^2} \langle \pi, C \rangle - nd(\log \sigma- \sigma^2), 
\]
the quadratically regularised optimal transport involves minimisation of 
\[
\pi \mapsto \langle \pi, C \rangle + \eps_Q \|\pi\|^2_F 
\]
over the set of hollow bistochastic matrices.
In view of the developments above, one can can rewrite the first minimisation problem, up to asymptotically negligible factors, as
\[
\max_{\pi\in B^\circ_n}  \frac{  -n H(\theta)  \| \pi\|^2_F }{K } -  \frac{1}{\sigma^2} \langle \pi, C \rangle - nd(\log \sigma- \sigma^2), 
\]
Hence, the above computations suggest the choice
\[
\eps_Q = \frac{n\sigma^2}{K} H(\theta)
\]
to make both problems nearly equivalent as $\sigma^2, \theta$ and $K$ are assumedly known and fixed. The claim of Theorem~\ref{thm: main} thus follows.

\subsection{Robustness}
It remains to provide one final proposition.
\begin{prop}
\label{prop: Robust}
Assume that $Y_i$ are as above, i.e., independent and normally distributed with covariance matrix $\sigma^2 I_d$. Only the mean vector depends on the cluster. 
Assume that there exist a constant $C_\mu$ such that 
\[
\max_{1\le k,k' \le K} \| \mu_k -\mu_{k'} \|\le C_\mu.
\]
Now consider corrupted observations
\[
\tilde Y_i = Y_i + \eta_i ,
\]
with $\eta_i \sim \mathcal{N}(0,\Sigma_i)$. The $\Sigma_i's$ can be different as long as  $\|\Sigma_i \|_2 \le C_\eta, \forall i$. 
Then, 
\[
\| \tilde Y_i - \tilde Y_j  \|^2= \|  Y_i -  Y_j  \|^2  +   \|\eta_i \|^2 +    \|\eta_j \|^2 + \Oh_p\left(\sqrt{d}\right).
\]
\end{prop}

\begin{proof}
It holds that for $i \neq j$, 
\[
\| \tilde Y_i - \tilde Y_j  \|^2= \|  Y_i -  Y_j  \|^2 + 2 \langle Y_i -  Y_j, \eta_i - \eta_j \rangle - 2\langle \eta_i, \eta_j \rangle +   \|\eta_i \|^2 +    \|\eta_j \|^2.
\]
Observe that 
\[
\expec \left[ \langle Y_i -  Y_j, \eta_i - \eta_j \rangle\right]=0 \quad \text{and} \quad \expec \left[\langle \eta_i, \eta_j \rangle\right] =0.
\]
Further, by von Neumann's trace inequality,
\[
 \expec \left[\langle \eta_i, \eta_j \rangle^2\right] = \tr [\Sigma_i\Sigma_j ] \le \sum_{k=1}^d  \lambda_k(\Sigma_i)\lambda_k(\Sigma_j ) \le C_\eta^2 d.
\]
Now, conclude by using 
\begin{align*}
\expec \left[ \langle Y_i -  Y_j, \eta_i - \eta_j \rangle^2\right] &= \expec \left[  (Y_i -  Y_j)^\top ( \eta_i - \eta_j)( \eta_i - \eta_j)^\top(Y_i -  Y_j)  \right]\\
&= \expec \left[  (Y_i -  Y_j)^\top ( \Sigma_i + \Sigma_j)(Y_i -  Y_j)  \right]\\
&\le \| \dot\mu \|^2 ( \|\Sigma_i\|_2 + \|\Sigma_j\|_2) + 2\sigma^2 (\tr \Sigma_i+ \tr \Sigma_j). \qedhere
\end{align*}
\end{proof}

In words, the previous proposition states that the cost matrix of the perturbed points\textemdash giving rise to anisotropic Gaussian mixtures, is equal to that of the isotropic case  
up to a controlled error and a rank-one matrix.

\section{Further remarks and conclusion}
\label{sec: FurRem}

\subsection{Entropy regularised self-transport}
It is natural to ask whether the entropic regularised optimal transport problem is equally good for the same problem.
Remark that $Z_{i,k} Z_{j,k}$ is a product of independent Bernoulli distributions for $j\neq i$ 
and thus 
\[
Z_{i,k} Z_{j,k}  \theta_k^{-1} =  \theta_k + O_p \left(\sqrt{\theta^2_k (1-\theta_k^2) }\right) = \theta_k\left( 1 + O_p\left( \sqrt{ 1-\theta_k^2 } \right)\right) \qquad (i\neq j),
\]
where $\Oh_p$ is here meant for  $\theta_k $ small.
Therefore, as $ N_k-1= n\theta_k(1 + \oh_p(1))$,
\begin{align*}
H(\pi):&=  - \sum_{i,j} \sum_k \1_{i\neq j}Z_{ik}Z_{jk} (N_k-1)^{-1}  \log \left(1  + \Oh_p\left(\sum_k \sqrt{ 1-\theta_k^2 }\right) \right).
\end{align*}
This quantity is more difficult to parse and an optimal choice of the regularisation parameter does not seem obvious.  
Note however that, contrarily to the case of the quadratic regularisation, as the optimal coupling is not sparse, one cannot expect an entropy regularised transport plan to be close to the matrix $\pi$ defined above and depending on $Z$.

\subsection{Non-Gaussian mixtures}
It is also natural to wonder whether the approach works for non-Gaussian mixture models.
The answer is not obvious and given the good results of general pseudo-likelihood estimators in certain cases, one might hope for some stability in the case of slight model misspecification.
Another possibility would be to consider transformations of the distance matrix that mimic the log-likelihood of the mixture components. It seems unlikely that other models will give rise to the same nice properties as in the Gaussian case.
Investigating generalisations is left for future research.

\subsection{Conclusion}
We have shown a link between quadratic-regularised (hollow) self-transport (QOT) and maximum likelihood estimation in Gaussian mixture models. 
Also, the choice of hollow bistochastic matrices is important for both approximation and robustness reasons. 
This links shows that applying spectral clustering techniques on the obtained transport plan is a statistically principled approach. 
The method is also easy to apply as it only requires choosing one parameter.
This also suggests that QOT might be used to complement the current EM algorithms or help find a good initialisation for other methods. As the spectrum of $\pi^*$ contains relevant information, it might be used to complement other statistical procedures.

\printbibliography

\end{document}